\pdfoutput=1

\documentclass[reqno,12pt]{article}

\usepackage[T1]{fontenc}
\usepackage[utf8]{inputenc}

\usepackage{amsmath}
\usepackage{amsthm}
\usepackage{amssymb}
\usepackage{authblk}
\usepackage{enumerate}
\usepackage{fullpage}
\usepackage{enumitem}
\usepackage{picture}
\usepackage{xcolor}
\usepackage{graphicx}

\newcommand{\beq}[1]{\begin{equation}\label{#1}}
\newcommand{\eeq}{\end{equation}}
\newcommand{\beqa}{\begin{eqnarray*}}
\newcommand{\eeqa}{\end{eqnarray*}}

\newcommand{\subs}{\subseteq}

\newcommand{\lf}{\lfloor}
\newcommand{\rf}{\rfloor}
\newcommand{\lc}{\lceil}
\newcommand{\rc}{\rceil}
\renewcommand{\le}{\leqslant}
\renewcommand{\ge}{\geqslant}

\newtheorem{theorem}{Theorem}
\newtheorem{lemma}[theorem]{Lemma}

\newtheorem{cor}[theorem]{Corollary}

\def\N{\mathbb N}
\def\kmin{k_{\textrm{min}}}

\author{
Ma\l gorzata Bednarska-Bzd\k{e}ga
    \thanks{Faculty of Mathematics and Computer Science, Adam Mickiewicz University in Pozna\'n, Poland. Email: {\tt mbed@amu.edu.pl}.} \ \ \
Michael Krivelevich
    \thanks{School of Mathematical Sciences, Raymond and Beverly Sackler Faculty of Exact Sciences, Tel Aviv University, Tel Aviv 6997801, Israel. Email: {\tt krivelev@post.tau.ac.il}. Supported in part by USA-Israel BSF grant 2014361.}
\and Viola M\'esz\'aros
    \thanks{Bolyai Int\'ezet, Aradi v\'ertan\'uk tere 1, Szeged 6720, Hungary. Email: {\tt viola@math.u-szeged.hu}. Supported by Swiss National Science Foundation grants 200020-162884 and 200020-144531.} \ \ \
Cl\'ement Requil\'e
    \thanks{Institute for Algebra, Johannes Kepler Universit\"at Linz, Austria. Email: {\tt clement.requile@jku.at}. Supported by the Austrian Science Fund (FWF) grant F5004. 
    The work presented in this paper was carried out while the author was affiliated with the Institut f\"ur Mathematik und Informatik, Freie Universit\"at Berlin, and Berlin Mathematical School, Germany, and supported by the FP7-PEOPLE-2013-CIG project CountGraph (ref. 630749).}
}

\begin{document}

\title{Proper colouring Painter-Builder game}

\maketitle

\begin{abstract}
    We consider the following two-player game, parametrised by positive integers $n$ and $k$. 
    The game is played between Painter and Builder, alternately taking turns, with Painter moving first. 
    The game starts with the empty graph on $n$ vertices. 
    In each round Painter colours a vertex of her choice by one of the $k$ colours and Builder adds an edge between two previously unconnected vertices. 
    Both players must adhere to the restriction that the game graph is properly $k$-coloured. 
    The game ends if either all $n$ vertices have been coloured, or Painter has no legal move. 
    In the former case, Painter wins the game, in the latter one Builder is the winner. 
    We prove that the minimal number of colours $k=k(n)$ allowing Painter's win is of logarithmic order in the number of vertices $n$. 
    Biased versions of the game are also considered.
\end{abstract}

\newpage


\section{Introduction}

We study the following perfect information game.
Let $n\ge 2$ and $k$ be natural numbers. 
The {\sl Painter-Builder $k$-colouring game} on $[n]=\{1,2,\ldots,n\}$ is played between Painter (who starts the game) and Builder. 
The players take turns to create/grow a graph $G$, the {\em game graph}, with a $k$-colouring of its vertices in the following way. 
The vertex set of $G$ is $[n]$ and at the beginning there are no edges.
In every round Painter colours a previously uncoloured vertex with one of $k$ colours. 
Then Builder adds to $G$ an edge between two previously unconnected vertices.
Throughout the game, both players must adhere to the restriction that the game graph is properly $k$-coloured.
The game ends if either all vertices are coloured or Painter has no legal move. 
In the first case, Painter wins the game (and the game lasts $n$ rounds), and in the latter one Painter is the loser. 
Observe that Painter loses the game if and only if at some point of the game there is an uncoloured vertex $v$, having neighbours in each one of the $k$ colours.

The game belongs to a wide class of Builder-Painter games, in which Builder reveals edges or vertices of a graph (which is the board of the game) and Painter colours edges or vertices of the board. 
The aim of Painter is to avoid building a monochromatic copy of a graph from a given graph family. 
There are different variants of the game, but they are mostly linked to the so called  online size Ramsey number. 
The online size Ramsey number  has been investigated in numerous papers starting with \cite{B83, B93} by Beck. 
Builder selects one edge per turn, then Painter has to colour it with one of $k$ colours. 
Builder is trying to force Painter to complete a monochromatic copy of some fixed graph $H$ in as few rounds as possible on an unbounded vertex set. 
The {\em online size Ramsey number} of $H$ is the minimum number of rounds until Builder achieves his goal assuming both players play optimally. 
Most of the results were obtained for the case of $k=2$ colours. 
In particular, Conlon \cite{C09} proved that for infinitely many values of $l$ the online size Ramsey number of $K_l$ is exponentially smaller than the size Ramsey number of $K_l$. 
Grytczuk, Kierstead and Pra{\l}at~\cite{GKP03} investigated the online size Ramsey number for paths and for some other graphs.
Grytczuk, Ha{\l}uszczak and Kierstead~\cite{GHK04} described yet another variant of the game: Builder selects an edge and Painter colours it but there is a restriction. 
While Builder tries to force Painter to build some fixed monochromatic graph $H$, Builder cannot choose edges arbitrarily. 
He has to keep the constructed graph in a prescribed class of finite graphs $\cal G$ throughout the game. 
The task is to determine the winner for a given pair $H, {\cal G}$. 
They showed  that for the case of two colours, Builder has a winning strategy for any $k$-colourable graph $H$ in the class of $k$-colourable graphs.
Belfrage, M\"utze and Sp\"ohel \cite{BMS11} used Builder-Painter games as a tool for studying a probabilistic one-player online Ramsey game. 
In their version, Builder reveals one edge per turn and he has to abide by the restriction that the edge density of the already revealed graph must not exceed a given bound $d$.

Let us stress that in all above versions of the Builder-Painter game, Painter colours edges and not vertices.
However, the vertex colouring games have also been studied. 
Motivated by a problem of an online colouring random process, M{\"u}tze, Rast and Sp{\"o}hel \cite{MRS12} and M{\"u}tze and Sp{\"o}hel  \cite{MS11}  considered the following game.
Given $k\in \N$, $d>0$, a graph $H$ and an infinite set of vertices, in each round Builder draws edges connecting a previously unselected vertex $v$ to some of the previously selected vertices, and Painter has to respond by colouring $v$ with one of $k$ colours. 
Builder is not allowed to draw an edge that would create a subgraph with edge density greater than $d$. 
Builder's aim is to force Painter to create a monochromatic copy of $H$.

The above  Builder-Painter game has a different focus from the $k$-colouring game in the present paper.
One important characteristic of our game is that it only considers proper colourings.
Moreover, subject to the constraint that the graph be properly colored, Builder can draw any edge and Painter can colour any uncoloured vertex, even one not incident to any edge selected by Builder. 
The constraint of keeping the edge density low is voided, however, we will consider a version of the Painter-Builder $k$-colouring game with the constraint of keeping the graph created by Builder 2-colourable. 
We will show that, somewhat surprisingly, even then Builder can force Painter to use many colours. 
More precisely, Painter needs $\Theta(\log n)$ colours to colour all $n$ vertices properly.

Let $\kmin(n)$ denote the minimum number of colours such that Painter has a winning strategy in the $k$-colouring game on $[n]$. 
Clearly $\kmin(n)\le n$ since Painter trivially wins the $n$-colouring game on $[n]$.

Our main result asserts that $\kmin(n)=\Theta(\log n)$.

\begin{theorem}\label{logcol}
    For every $n>10^8$,
    $$
        0.01\log_2 n < k_{\min}(n)\le \log_2 n+1.
    $$
    Moreover, if $k\le 0.01\log_2 n$, then Builder has a strategy to win the $k$-colouring game on $[n]$, while keeping the game graph bipartite all game long.
\end{theorem}

Let us emphasise that we made no effort to optimise the constants in the theorem.
The proof of the theorem will be discussed in the following two sections. 
We present a random argument for the existence of a winning strategy for Painter in Section \ref{sec:paint}, and we exhibit a winning strategy for Builder in Section \ref{sec:build}. 
Section \ref{sec:bias} is devoted to the biased version of the Painter-Builder $k$-colouring game, where Builder claims $b\ge 1$ edges each time.
At every moment of the game by the \emph{game graph} we mean the graph with the vertex set $[n]$ and the edge set consisting of all edges drawn by Builder so far.

\section{Painter's strategy}\label{sec:paint}

In this section, we describe a greedy-random strategy for Painter for the $(\lf\log_2 n\rf +1)$-colouring game on $[n]$.
We prove that with this strategy she wins with positive probability against any strategy of Builder.
Suppose the colours are the integers $1, 2, 3,\dots, \lf\log_2 n\rf +1$. The strategy of Painter is as follows.

She starts by colouring a vertex with colour 1. Then, every time Builder adds a new edge $e=uv$ to the graph, Painter responds by selecting a vertex to colour in the following way:

\begin{itemize}
    \item If both endpoints $u,v$ are uncoloured, then Painter chooses one of $u,v$ at random, with probability $1/2$.
    \item If exactly one of $u,v$  is uncoloured, then Painter selects it.
    \item If both endpoints are coloured, then Painter chooses an arbitrary uncoloured vertex.
\end{itemize}

After having chosen which vertex to colour, Painter colours it with the smallest available colour (that is with the smallest colour currently not assigned to its neighbours). 
If there is no such colour among all colours available for the game, then Painter forfeits the game.

To analyse this strategy, let us denote by $A_v$ the (random) event that at some point in the game, a vertex $v\in [n]$ is still uncoloured, and has degree $\lf\log_2 n\rf + 1$. The probability of such an event is  bounded as follows:
$$
    {\mathbb P}(A_v) \leq \left( \frac{1}{2} \right)^{\lf \log_2n\rf+1}<\frac{1}{n}.
$$
Indeed, for $v$ to remain uncoloured with its degree already that large, every time the edge connecting $v$ to a vertex $w$ of its neighbourhood is added by Builder, Painter selected $w$ and he did so with probability at most $1/2$, independently of previous choices.

Observe that if the event $A_v$ did not happen during the game, then $v$ was chosen for colouring  when at most $\lf\log_2n\rf$ of its neighbours were coloured, thus there was a colour available for $v$ among the first $\lf\log_2n\rf+1$ colours. 
Hence Builder can win the game only if at least one such event happens. 
The probability that Builder wins the game is then at most ${\mathbb P}(\cup_{v\in [n]} A_v)<1$. 
Since the game is a perfect information game with no chance moves, this means that Builder has no winning strategy, thus indicating Painter's win. 
Hereby we obtain the upper bound of Theorem~\ref{logcol}.

\section{Builder's strategy}\label{sec:build}

We say that  there is {\sl a waiting-room $(A,B)$} in the $k$-colouring game on $[n]$ if:
\begin{itemize}[label=-]
    \item $A,B\subs V(G)$ are two independent sets,
    
    \item $|A|=|B|\ge 0.1n/k$,
    
    
    \item $G[A \cup B]$ is a perfect matching,
    
    \item all vertices in $A$ are coloured with the same colour.
\end{itemize}
If there is a waiting-room $(A,B)$, we say that Builder {\sl adds an edge to the waiting-room} if he claims a pair $ab$ with $a\in A$ and $b\in B$.
Observe that the game lasts at most $n$ rounds.
Furthermore, Painter cannot colour any vertex in $B$ with the colour of vertices in $A$.
Hence, every edge added between $A$ and $B$ will be properly coloured.
Therefore, if $|A|\cdot |B| > n$, then Builder will be able to add an edge to the waiting-room in every round, until the game ends. 
This observation will prove useful later, when we say that Builder can just wait in the waiting-room until enough vertices in a target set $C$ disjoint from $A\cup B$ have been coloured by Painter.
We would like to stress that the vertices of the waiting-room will not be used to enforce Painter to use many colours, but this structure will only be utilised to force Painter to colour many vertices outside of it.

First, we prove that Builder can create a waiting room quickly.

\begin{lemma}\label{waitingroom}
    Suppose that $n > 50$ and $k\ge 2$.
    Builder has a strategy ensuring that, after at most  $0.2n$ rounds of the $k$-colouring game on $[n]$, the game graph $G$ contains a waiting-room $(A,B)$ and $G$ is a union of disjoint paths of length at most two.
\end{lemma}

\begin{proof}
    In the first $\lc 0.1n\rc$ rounds Builder creates a matching. 
    After that and since there are at most $k$ colours present on the graph, there is a set $A=\{v_1,\ldots,v_t\}$ of vertices which have exactly the same colour, for $t=\lc 0.1n/k\rc\le \lc 0.05n\rc$. Now, for $i=1,\ldots,t$, and in every one of the next $t$ rounds, Builder selects a non-coloured and isolated vertex $u_i$ and claims the pair $v_iu_i$ (see Figure \ref{fig:waiting-room}).
    \begin{figure}
    \centering
            \includegraphics[scale=1]{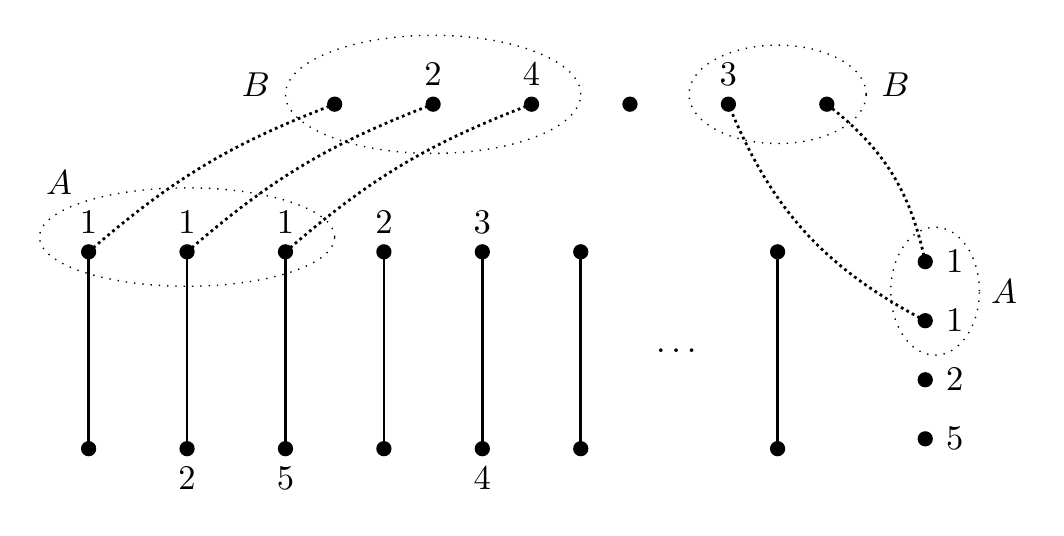}
    \caption{{\sl Waiting-room} $(A,B)$. First Builder creates a matching (thick edges), then connects (via the densely dotted edges) each vertex of $A$, the set of vertices with colour 1, to a different uncoloured and isolated vertex which will be added to the set $B$.}
    \label{fig:waiting-room}
    \end{figure}
    Let us stress that for $n>50$ such an isolated vertex always exists since $3(\lc 0.1n\rc +t)<n$. It is not hard to verify that $(A,B)$, with  $B=\{u_1,\ldots,u_t\}$, forms a waiting-room. 
    Furthermore the graph of the game consists of disjoint paths of length at most 2, and it took Builder $\lc 0.1n\rc +t$ rounds to create them.
    Thus, for $n>50$, the number of rounds is less than $0.2n$.
\end{proof}

Given $t\ge 0$, we say that a vertex has a {\em $t$-colour-neighbourhood} if at least $t$ of its neighbours have pairwise distinct colours. 
We say that the vertices $u,u'$ are in the same {\em $t$-colour-class} if they have $t$-colour-neighbourhoods with exactly the same set of colours. 
Finally, we say that a vertex has a {\em strict $t$-colour-neighbourhood} if it has a $t$-colour-neighbourhood but no $(t+1)$-colour-neighbourhood.

\begin{lemma}\label{excludecolour}
    Suppose that $0\le t<k$ and that after some rounds of the game there is a waiting-room $(A,B)$ and a set $V_t$ of uncoloured vertices, each having a $t$-colour-neighbourhood, with $|V_t|>2\binom kt$ and $|V_t|>1000$.
    Furthermore, assume that in the game graph $G$ the component of every vertex of $V_t$ is a tree and that no two vertices of $V_t$ are in the same component; in addition, if $G$ contains a cycle, then all of its edges are between $A$ and $B$.

    Then Builder has a strategy ensuring that after some additional rounds the following holds.
    \begin{itemize}[label=-]
        \item There is a set $V_{t+1}$ of uncoloured vertices with $|V_{t+1}|>0.001 |V_t|$,  each having a $(t+1)$-colour-neighbourhood.
        \item No two vertices of $V_{t+1}$ are in the same component.
        \item The component of every vertex from $V_{t+1}$ is a tree.
        \item If there is a cycle in the game graph, then all of its edges are between $A$ and $B$.
    \end{itemize}
\end{lemma}

\begin{proof}
    Suppose the assumptions of the lemma hold, at some moment of the game, for a set $V_t$ of $n_t$ vertices. 
    If there is a set $V'\subs V_t$ such that $|V'|>0.001n_t$ and every element of $V'$ has a $(t+1)$-colour-neighbourhood, then clearly our assertion follows with $V_{t+1}=V'$. 
    Therefore we assume further that there is no such such set $V'$ and that there is a set $V'_t\subs V_t$ of at least $0.999n_t$ elements, each one of them having a strict $t$-colour-neighbourhood.

    First and for as long as possible, Builder consecutively matches vertices of $V'_t$ of the same $t$-colour-class. 
    Suppose that it takes Builder $s$ rounds to create a matching $M$ in such a way. 
    Given the $s$ moves of Painter and the parity of the size of every $t$-colour-class, we have $s=|M|\ge \big(|V'_t|-\binom kt-s\big)/2$.
    Thus $|M|\ge \big(|V'_t|-\binom kt\big)/3>0.16n_t$, because $|V'_t|\ge 0.999n_t$ and $n_t>2\binom kt$. 
    Clearly at most $s$ vertices covered by $M$ are coloured so there is a set $M'$ of at least $s/2 = |M|/2 > 0.08 n_t$ edges of $M$ such that every edge in $M'$ has at most one endpoint coloured. 
    Now if $j>0.001n_t$ of these edges, say $v_1v'_1,\ldots,v_jv'_j$, have exactly one endpoint coloured, i.e. $v_1,\ldots,v_j$, then we set $V_{t+1}=\{v'_1,\ldots,v'_j\}$. 
    It is then not hard to verify that $|V_{t+1}|>0.001n_t$, that every vertex $v'_1,\ldots,v'_j$ has a $(t+1)$-colour-neighbourhood and that no two of them belong to the same component. 
    Moreover, throughout the considered $s$ rounds, Builder has never connected two vertices from the same component.

    Assume now that at least $|M'|-0.001n_t>0.07n_t$ edges of $M'$ have no coloured endpoint, and denote these edges by $u_1u'_1,\ldots,u_mu'_m$. 
    We have $m>0.07n_t$. 
    We consider two cases.

    \begin{itemize}
        \item {\sl Case 1.} There are at least $m'=\lceil m/3\rceil$ edges among $u_1u'_1, \ldots,u_mu'_m$, denoted by $u_1u'_1,\ldots,u_{m'}u'_{m'}$, such that all their endpoints $u_1$, $u'_1$, $\ldots$, $u_{m'}$, $u'_{m'}$ are in the same $t$-colour-class.
    \end{itemize}

    Now Builder adds edges to the waiting-room until Painter colours at least one endpoint of $\lfloor m'/6\rfloor$ edges from $\{u_1u'_1,\ldots,u_{m'}u'_{m'}\}$. 
    Then the vertices, say, $u_1,\ldots,u_{m''}$ are coloured, with $m''=\lfloor m'/6\rfloor$.
    Observe that $m''> 0.003n_t$, since $n_t>1000$. 
    At this very moment, some vertices $u_i$, with $m''+1\le i\le m'$, may have a $(t+1)$-colour-neighbourhood. 
    Notice that if there are at least $\lfloor (m'-m'')/2\rfloor$ of those vertices $u_i$, then they form the desired set $V_{t+1}$ since $\lfloor (m'-m'')/2\rfloor>m'/3\ge 0.006n_t$. 
    Suppose then that it is not the case, i.e.~at least $(m'-m'')/2>2m''$ vertices, say $u_{m''+1},\ldots,u_{m''+j}$ with $j > 2m''$, still have strict $t$-colour-neighbourhoods. 
    Then every $u_i$, with $1\le i\le m''$, is coloured with a colour which does not appear in the neighbourhood of any vertex in $\{u_{m''+1},u_{m''+2},\ldots, u_{3m''}\}$.

    Then in the consecutive $m''$ rounds, Builder creates a matching between $\{u_1,\ldots,u_{m''}\}$ and
    $\{u_{m''+1},u_{m''+2},\ldots,u_{3m''}\}$. 
    Without loss of generality, we can assume that this matching is $\{u_1u_{m''+1},u_2u_{m''+2},\ldots,u_{m''}u_{2m''}\}$. 
    Now in the meantime, Painter has coloured at most $m''$ vertices in the set \\
    $
        \{u_{m''+1},u_{m''+2},\ldots,u_{2m''}\}\cup \{u'_{m''+1},u'_{m''+2},\ldots,u'_{2m''}\}.
    $
    So for at least $m''/2$ indices $i$ with $m''+1\le i\le 2m''$, we have that  $u_i$ or $u'_i$ is uncoloured. 
    Following this observation, we define the set

    \begin{align*}
        V_{t+1} := & \{u_i\colon m''+1\le i\le 2m'' \text{ and $u_i$ is uncoloured}\}\\
        & \cup \{u'_i\colon m''+1\le i\le 2m'' \text{ and $u_i$ is coloured, while $u_i'$ is uncoloured}\}.
    \end{align*}
    It holds that $|V_{t+1}|\ge m''/2> 0.0015n_t$, and every vertex in $V_{t+1}$ has a $(t+1)$-colour-neighbourhood.
    Furthermore no two vertices of $V_{t+1}$ are in the same component. 
    Observe also that Builder never drew an edge between two vertices of the same component, except for the edges added to the waiting-room. Hence $V_{t+1}$ is the desired set.

    \begin{itemize}
        \item {\sl Case 2.} Every subset of $\{u_1u_1', \ldots, u_mu'_m\}$ with the vertices in the same $t$-colour-class has less than $m/3$ elements.
    \end{itemize}

    Recall that each edge of the matching $M$ connected two vertices with identical $t$-colour-classes. 
    Hence there exists a subset of $m_0$ edges of the set\\
    $\{u_1u'_1, \ldots, u_mu'_m\}$, say $u_1u'_1$, $\ldots$, $u_{m_0}u'_{m_0}$, such that $m/6\le m_0\le m/3$ and everyone of the endpoints $u_1,u'_1$, $\ldots$, $u_{m_0},u'_{m_0}$ has a different $t$-colour-class than the one of each vertex in $\{u_{m_0+1},u_{m_0+2},\ldots,u_{m}\}$.
    Indeed, if there exists a $t$-colour-class $C$ containing between $m/6$ and $m/3$ edges, then the edges in $C$ form the desired set. 
    Otherwise, each $t$-colour-class contains less than $m/6$ edges, and it is easy to see that the union of some $t$-colour-classes has size between $m/6$ and $m/3$, thus forming the desired set of $m_0$ edges.
    Observe now that $m_0> 0.01n_t$ and notice that for every $u_i$ and $u_j$, with $1\le i\le m_0$ and $m_0+1\le j\le m$, $u_i$ has a colour in its neighbourhood which does not appear in the neighbourhood of $u_j$ (and of $u'_j$). 
    This follows from the fact that $u_i$ and $u_j$ have distinct, strict $t$-colour-neighbourhoods.
    
%

    Then in the next $m_0$ consecutive rounds, Builder creates a matching $w_1u_{m_0+1}$, $w_2u_{m_0+2}$,$\ldots$, $w_{m_0}u_{2m_0}$ such that for every $1\le i\le m_0$, the vertex $w_i$ is a coloured neighbour of $u_i$ and the colour of $w_i$ does not appear in the neighbourhood of $u_{m_0+i}$.            
    Since the vertices $u_{1},u_{2},\ldots,u_{m_0}$ have disjoint neighbourhoods, we have $w_i\neq w_j$ for distinct $i,j$. 
    In the meantime, Painter has coloured at most $m_0$ vertices in the set $\{u_{m_0+1},u_{m_0+2},\ldots,u_{2m_0}\}\cup \{u'_{m_0+1},u'_{m_0+2},\ldots,u'_{2m_0}\}$. 
    So for at least $m_0/2$ indices $1\le i\le m_0$, either $u_{m_0+1}$ or $u'_{m_0+1}$ is uncoloured. 
    Let us define the set
    \begin{align*}
        V_{t+1} &:= \{u_{m_0+i}\colon 1\le i\le m_0 \text{ and $u_{m_0+i}$ is uncoloured}\}\\
                & \cup \{u'_{m_0+i}\colon 1\le i\le m_0 \text{ and $u_{m_0+i}$ is coloured, while $u'_{m_0+i}$ is uncoloured}\}.
    \end{align*}
    
    One can easily verify that $|V_{t+1}|\ge m_0/2> 0.005n_t$, that every vertex in $V_{t+1}$ has a $(t+1)$-colour-neighbourhood and no two vertices of $V_{t+1}$ are in the same component. 
    Furthermore, Builder never connected any two vertices of the same component. 
    Hence, again $V_{t+1}$ is the desired set.
\end{proof}

We are now ready for the proof of the lower bound in Theorem~\ref{logcol}.
Let  $n>10^8$ and $k=\lfloor 0.01\log_2 n\rfloor$.
Builder proceeds as  follows. 
For the first $\lf 0.2n\rf$ rounds, Builder plays so that the resulting graph $G$ consists of disjoint paths of length at most 2 and admits a waiting-room $(A,B)$. 
This is possible in view of Lemma \ref{waitingroom}. 
At this very moment, the assumptions of Lemma \ref{excludecolour} hold for $t=0$, where $V_0$ is the set of all uncoloured and isolated vertices of $G$.
Indeed, we have $|V_0|\ge n-3\cdot 0.2n>0.2n>1000$.

Now, based on Lemma \ref{excludecolour}, we infer by an easy induction argument that Builder can play so that for $j=1,\ldots, k$ there is a moment of the game when there is a set $V_j$ of vertices which satisfies all the following conditions.

\begin{itemize}[label=-]
    \item The game graph is bipartite.
    \item All vertices of $V_j$ are uncoloured.
    \item Every vertex of $V_j$ has a $j$-colour-neighbourhood.
    \item $|V_j| > 0.001^j|V_0| > 0.001^k\cdot 0.2n > 2^{k+1} > 2\binom kj$ and $|V_j|>1000$.
\end{itemize}

The last two inequalities hold for $k\le 0.01\log_2 n$ and $n>10^8$.
We conclude that in the set $V_k$ there is an uncoloured vertex with a $k$-colour-neighbourhood, which means that Builder wins the game. 
This implies that $\kmin(n)>0.01\log_2 n$.
As mentioned in the introduction, this result is quite surprising as the game graph remains bipartite throughout the game.

\section{The biased game}\label{sec:bias}

In this section, we consider the {\em biased} version of the $k$-colouring game on $[n]$ between Builder and Painter.
Given positive  integers $p$ and $b$, in each round of the $(p : b)$ Painter-Builder $k$-colouring game on $n$ vertices, Painter colours exactly $p$ vertices and Builder adds exactly $b$ edges to the graph, apart possibly from the very last moves of the players. 
If immediately before Painter's last move there are less than $p$ uncoloured vertices, then she colours as many of them as she can and the game ends. 
If  immediately before Builder's last move there are less than $b$ unselected pairs, then he selects them all and Painter colours as many vertices as possible and the game ends.
Painter wins if and only if at the end of the game all $n$ vertices are (properly) coloured.

Let us first observe that in the $(2 : 1)$ Painter-Builder game two colours suffice for Painter to win, by simply making sure in each round that the last edge inserted by Builder carries both colours. 
Note that in a round when only one endpoint of the inserted edge is uncoloured, Painter can choose and colour arbitrarily an uncoloured, isolated vertex, in addition to colouring properly the uncoloured vertex of the newly added edge.
As after each of Painter's moves all non-isolated vertices are coloured, no vertex has neighbours in both colour classes. 
Hence, two colours guarantee Painter's win.

We then focus on the situation where the bias is in favour of Builder. 
More explicitly, we will consider the $(1 : b)$ Painter-Builder $k$-colouring game. 
Let $\kmin(b,n)$ denote the minimal number of colours such that Painter has a winning strategy in the $(1 : b)$ colouring game on $[n]$. 
Clearly $\kmin(b,n)\le n$ for every $b$.
Before we state and prove the main result of this section, we show that in the biased game Builder can build a big clique.

\begin{lemma}\label{bigclique}
    Define the sequence $(n_i)_{i=0}^{\infty}$ by
    $$
        n_0=n,\quad n_{i+1}=n_i-\left\lceil \frac{n_i-1}{b}\right\rceil-1\,.
    $$
    Let $t=\max\{i:n_i>0\}$ and $0\le j\le t$.

    Then Builder can play so that immediately before some move of Painter there are disjoint vertex sets $K_j$  and $V_j$ such that $K_j$ spans a clique of size $j$ in the game graph $G$, $|V_j|\ge n_j$, no vertices of $V_j$ are coloured, and $uv\in E(G)$ for every $u\in K_j$ and $v\in V_j$.
\end{lemma}

\begin{proof}
    We proceed by induction on $j$.
    Start trivially with $V_0=[n]$, $K_0=\emptyset$. 
    For the induction step, suppose that $0\le j\le t-1$ and before some round of the game there are (possibly empty) vertex sets $K_j=\{x_1,\ldots,x_j\}$ and $V_j$ such that $K_j$ spans a clique in the game graph, $|V_j|\ge n_j$, no vertices of $V_j$ are coloured, $V_j$ is disjoint from $K_j$, and $uv\in E(G)$ for every $u\in K_j$ and $v\in V_j$.

    Since $j\le t-1$, the set $V_j$ is nonempty by the definition of $t$. 
    Furthermore without loss of generality we can assume that $|V_j|=n_j$.
    Builder then chooses an arbitrary vertex $x_{j+1}\in V_j$ and in consecutive rounds, for as long as possible, takes $b$ uncoloured vertices in $V_j$ and connects them to $x_{j+1}$. 
    If eventually there are less than $b$ uncoloured vertices in $V_j$ not connected to $x_{j+1}$, then Builder connects all of them to $x_{j+1}$ and selects the remaining edges arbitrarily. 
    Let $V_{j+1}$ be the set of vertices connected to $x_{j+1}$ and not coloured after Builder has completed his task. 
    Update $K_{j+1}:=K_j\cup\{x_{j+1}\}$. 
    It is easy to verify that $K_{j+1}$ spans a clique, every vertex of $K_{j+1}$ is connected to every vertex of $V_{j+1}$ and
    $$
            |V_{j+1}|\ge  |V_j|-1-\left\lceil \frac{|V_j|-1}{b}\right\rceil=n_j-\left\lceil \frac{n_j-1}{b}\right\rceil-1=n_{j+1}\,.
    $$
    Thus the assertion follows.
\end{proof}

Here is the main result of this section.

\begin{theorem}\label{biased}
    For every $b\ge 2$ and $n\ge 2$
    $$
        \frac b2\ln\Big(\frac n{2b}+1\Big)< k_{\min}(b,n)\le \min\{\lc 2b\ln n\rc,n\}.
    $$
\end{theorem}

\begin{proof}
    In order to prove the lower bound, it is enough to show that Builder can create a clique on $\big\lf\frac b2\ln\big(\frac n{2b}+1\big)\big\rf+1$ vertices.          In view of Lemma~\ref{bigclique}, Builder can build a clique on $t+1$ vertices, where $t$ is defined in the lemma (the clique is spanned by $K_t$ and a vertex from $V_t$, defined in the lemma).
    By the definition of the sequence $(n_i)$ in the lemma, we have
    $$
            n_{i+1}>n_i\left(1-\frac1b\right)-2\,.
    $$
    Solving this recursive estimate gives
    $$
            n_t> (n+2b)\left(\frac{b-1}{b}\right)^t-2b\,.
    $$
    Since $b\ge 2$ and consequently $1-\frac1b> e^{-2/b}$, one can verify that the right-hand side of the above inequality is non-negative provided that $t\le \frac b2\ln\left(\frac{n}{2b}+1\right)$. 
    Therefore Builder has a winning strategy in the biased $k$-colouring game for every $k\le \frac b2\ln\big(\frac n{2b}+1\big)$.

    We now switch to Painter's side.
    Observe first that the probabilistic argument in Section~\ref{sec:paint} can be easily adapted for the biased case. 
    In a nutshell, the argument goes as follows. 
    Assume that Builder just placed edges $e_1,\ldots,e_b$ in his move. 
    Painter chooses one of their endpoints to colour at random, where the probability of choosing a particular vertex $u$ is $d/(2b)$, where $d$ is the degree of $u$ in the subgraph induced by $\{e_1,\ldots,e_b\}$. 
    By an analysis very similar to that in Section~\ref{sec:paint}, Painter wins with positive probability against any Builder's strategy if ${\mathbb P}(A_v)<1/n$, where $A_v$ is the event that at some point in the game, a vertex $v\in [n]$ is still uncoloured, and has degree at least $s:=\lc 2b\ln n\rc$. 
    Namely, by the definition of Painter's random strategy we obtain that for some positive integers $d_1,d_2,\ldots,d_j$ with $\sum_{i=1}^j d_i=s$ the following holds:
    $$
        {\mathbb P}(A_v) \leq \prod_{i=1}^j\left(1- \frac{d_i}{2b} \right)<e^{-\sum_{i=1}^j d_i/(2b)}=e^{-s/(2b)},
    $$
    and since $e^{-s/(2b)}\le 1/n$ for $s\ge 2b\ln n$, we obtain the desired bound on ${\mathbb P}(A_v)$.

    We conclude that Painter has a winning strategy in the biased $k$-colouring game for every $k\ge  \lc 2b\ln n\rc$.
    The estimate $\kmin(b,n)\le n$ is obvious, so the proof is complete.
\end{proof}

\begin{cor}
Let $\varepsilon<1$ be a positive constant and let $b=b(n)$.  Then
$$
    k_{\min}(b,n) =
        \begin{cases}
            \Theta(b\ln n)&\text{ if \,} 2\le b\le n^{1-\varepsilon},\\
            \Theta(n)&\text{ if \,} b=\Theta(n).
        \end{cases}
$$
\end{cor}

Theorem~\ref{biased} does not yield the order of $\kmin(b,n)$ when $b=n^{1-f(n)}$ and $f(n)$ is any positive function tending to 0 with $n\to\infty$ such that $f(n)=\omega(1/\ln n)$.
Then we have  $\Theta(f(n)b\ln n)\le\kmin(b,n)\le  \Theta(b\ln n)$. 
It would be interesting to find the order of $\kmin(b,n)$ in this case.

\section*{Acknowledgements}

The authors wish to thank Peleg Michaeli, for raising this problem in the course of the two {\em TAU-FUB Workshops on Positional Games} hosted successively by Tel Aviv University in 2015 and by the Freie Universit\"at Berlin in 2016, together with the organisers of the two workshops.
The authors also wish to thank the anonymous referees for their careful reading and helpful remarks.

\newpage

\bibliography{pcpbg_arxiv_v3}
\bibliographystyle{abbrv}

\end{document}